%% file: Paradigm_Artifacts.tex
\documentclass[11pt]{article}

\usepackage{amsmath,amsthm,amssymb}
\usepackage{graphicx}
\usepackage{enumerate}  
\usepackage[colorlinks=true]{hyperref}
\usepackage{csquotes}

\newtheorem{theorem}{Theorem}[section]
\newtheorem{corollary}[theorem]{Corollary}

\newtheorem{proposition}[theorem]{Proposition}

\theoremstyle{definition}
\newtheorem{definition}[theorem]{Definition}
\newtheorem{remark}{Remark}
\newtheorem{example}{Example}

\input{commands.tex}

\title{Limited data problems for the generalized Radon transform in $\rn$}

\author{
	J\"urgen Frikel\thanks{Department of Applied Mathematics and Computer Science, Technical University of Denmark, Matematiktorvet 303, 2800 Kgs Lyngby, Denmark. Email:  \texttt{jyfr@dtu.dk}}~ and Eric Todd Quinto\thanks{Department of Mathematics, Tufts University, Medford, MA 02155, USA. Email: \texttt{todd.quinto@tufts.edu}}
}

\begin{document}
\maketitle

%

%

\input{0_abstract.tex}

\input{1_introduction.tex}
\input{2_notation.tex}
\input{3_paradigm.tex}
\input{4_artifactsR2.tex}

\input{4_artifactsRn.tex}
\input{5_reduction.tex}

\input{6_conclusion.tex}
\input{7_acknowledgments.tex}
\input{8_appendix.tex}

\bibliographystyle{abbrv}
\bibliography{references}
\nocite{FrikelQuintoPreprint}

\end{document}

%% file: commands.tex
\mathchardef\ordinarycolon\mathcode`\:
\mathcode`\:=\string"8000\def\R{\mathbb{R}}
\begingroup \catcode`\:=\active
  \gdef:{\mathrel{\mathop\ordinarycolon}}
\endgroup

\renewcommand{\R}{\mathbb{R}}

\newcommand{\rr}{\mathbb{R}}

\usepackage{mathrsfs}

\newcommand{\Ac}{\mathcal{A}}

\newcommand{\Dc}{\mathcal{D}}
\newcommand{\Ec}{\mathcal{E}}
\newcommand{\Fc}{\mathcal{F}}

\newcommand{\Mc}{\mathcal{M}}

\newcommand{\Qc}{\mathcal{Q}}

\newcommand{\Sc}{\mathcal{S}}

\newcommand{\WF}{\mathrm{WF}}                         
\newcommand{\WFab}{\mathrm{WF_{(a,b)}}} 
\newcommand{\WFA}{\mathrm{WF_{A}}}

\newcommand{\norm}[1]{\left\lVert#1\right\rVert}      
\newcommand{\abs}[1]{\left|#1\right|}                 
\newcommand{\paren}[1]{\left(#1\right)}               
\newcommand{\bparen}[1]{\left[#1\right]}               
\newcommand{\sparen}[1]{\left\{#1\right\}}		      

\renewcommand{\d}{\,\mathrm{d}}						  

\renewcommand{\epsilon}{\varepsilon}
\renewcommand{\rho}{\varrho}
\newcommand{\vp}{\varphi}
\newcommand{\thperp}{{\theta^\bot}}

\renewcommand{\tilde}{\widetilde}

\newcommand{\tU}{\tilde{U}}

\newcommand{\WbdA}{W_{\operatorname{bd}(A)}}

\newcommand{\chib}{\chi_{B}}
\newcommand{\Pw}{\pi_{\omega}}

\newcommand{\LA}{\mathcal{L}_A}
\newcommand{\RmuA}{R_{\mu,A}}
\newcommand{\snm}{S^{n-1}}
\newcommand{\otp}{[0,2\pi]}
\newcommand{\rtwo}{{{\mathbb R}^2}}

\newcommand{\rn}{{{\mathbb R}^n}}
\newcommand{\Lab}{\mathcal{L}_{[a,b]}}

\newcommand{\st}{\hskip 0.3mm : \hskip 0.3mm}
\renewcommand{\th}{\theta}

\newcommand{\be}{\begin{equation}}
\newcommand{\ee}{\end{equation}}
\newcommand{\bea}{\begin{eqnarray}}
\newcommand{\eea}{\end{eqnarray}}
\newcommand{\bean}{\begin{eqnarray*}}
\newcommand{\eean}{\end{eqnarray*}}

\newcommand{\bel}[1]{\begin{equation}\label{#1}}

\newcommand{\eel}[1]{{\label{#1}\end{equation}}}

\newcommand{\inv}{^{-1}}

\newcommand{\sxr}{{S^1\times \mathbb{R}}}

\newcommand{\intt}{{\operatorname{int}}}
\newcommand{\cl}{{\operatorname{cl}}}
\newcommand{\bd}{{\operatorname{bd}}}

\newcommand{\smo}{\setminus\boldsymbol{0}}

\newcommand{\xio}{{\xi_0}}

\newcommand{\chiab}{\chi_{{[a,b]}}}
\newcommand{\chia}{\chi_{A}}
\newcommand{\chiar}{\chi_{A\times \mathbb{R}}}
\newcommand{\dx}{\mathrm{d}x}

\newcommand{\dr}{\mathrm{d}r}
\newcommand{\ds}{\mathrm{d}s}

\newcommand{\dphi}{\mathrm{d}\phi}
\newcommand{\Om}{\Omega}

\newcommand{\Vc}{\mathcal{V}}
\newcommand{\VA}{\mathcal{V}_{A}}
\newcommand{\VintA}{\mathcal{V}_{\operatorname{int}(A)}}

\newcommand{\VrA}{{V}^R_{A}}

\newcommand{\Vrab}{{V}^R_{[a,b]}}

\newcommand{\omo}{\om_0}

\newcommand{\lamo}{\lambda_0}
\newcommand{\lamon}{\lambda_1}

\newcommand{\Rmu}{R_\mu}
\newcommand{\Rmuab}{R_{\mu,[a,b]}}
\newcommand{\Rstn}{R^*_\nu}
\newcommand{\Kvp}{\mathcal{K}_\varphi}

\newcommand{\Lvp}{\mathcal{L}_{\varphi}}

\newcommand{\MB}{\mathcal{M}_B}
\newcommand{\Mst}{\mathcal{M}^\ast}

\newcommand{\om}{\omega}

%% file: 0_abstract.tex
\begin{abstract}
We consider the generalized Radon transform (defined in terms of
smooth weight functions) on hyperplanes in $\rn$.
We analyze general filtered backprojection type reconstruction methods
for limited data with filters given by general pseudodifferential
operators. We provide microlocal characterizations of
visible and added singularities in $\rn$ and define modified versions of
reconstruction operators that do not generate added artifacts. We
calculate the symbol of our general reconstruction operators as
pseudodifferential operators, and provide conditions for the filters
under which the reconstruction operators are elliptic for the visible
singularities.  If the filters are chosen according to those
conditions, we show that almost all visible singularities can be
recovered reliably. Our work generalizes the results for the classical 
line transforms in $\rtwo$ and the classical reconstruction operators (that use specific filters).
In our proofs, we employ a general paradigm that is based 
on the calculus of Fourier integral operators. Since this technique 
does not rely on explicit expressions of the reconstruction operators,
it enables us to analyze  more general imaging situations.
\\medskip
\end{abstract}


%% file: 1_introduction.tex
\section{Introduction}
\label{sec:introduction}


In this article, we analyze the limited data problem for the
generalized Radon transform integrating over hyperplanes in $\rn$
using microlocal analysis.  The integration along those hyperplanes is
performed with respect to some weight functions that might depend on
both the hyperplane and the point on the hyperplane. By considering
more general weight functions, we aim to analyze a wider class of
imaging applications, including emission tomography (such SPECT and
some models of PET) where the attenuated Radon transform is used to
model the measurement process, cf.\ \cite{Natterer86,Natterer01}, 
as well as possible future applications. Of course, our setup
also includes the results for the classical Radon transform with
constant weight.

 Our setup is as follows: 
Let $(\om,s)\in \Xi:=\snm\times \rr$, then we consider the hyperplanes
\bel{def:H} H(\om,s) = \sparen{x\in \rn\st x\cdot \om = s}\ee
perpendicular to $\om$ containing the point $s\om$, 
i.e., $ H(\om,s)$ is $s$ directed units from the origin (in the
direction of $\om$ if $s\geq 0$ and in the opposite direction if
$s<0$).  Let $\mu:\snm\times \rn\to \rr$ be a smooth nowhere zero
weight, then we define the generalized Radon transform 
\bel{def:Rmu-rn} \Rmu f(\om,s) = \int_{x\in H(\om,s)} f(x)
\mu(\om,x) \d x,\ee where $\d x$ is the Lebesgue measure on the hyperplane
$H(\om,s)$.  We also define a generalized  dual operator (or the backprojection 
operator) with respect to an arbitrary smooth weight $\nu=\nu(\om,x)$ as 
\bel{def:Rstn-rn}\Rstn g(x) = \int_{\om\in\snm} g(\om,x\cdot\om)\nu(\om,x)\d \om.\ee 
Note that this covers both standard cases when $\nu=\mu$ and so $\Rstn$ is the
adjoint operator $(\Rmu)^*$ and the case when $\nu=1/\mu$ which is considered
by some authors including Beylkin. We discuss these cases further in 
Remark \ref{remark:weights}. Moreover, we note that the above transforms are both defined and weakly
continuous for classes of distributions \cite{He:RT2011}.

%


Many inversion formulas have been proven for the classical Radon transform
($\mu\equiv 1$) \cite{Natterer86}, and invertibility of the generalized Radon transform $\Rmu$ has
been well studied (e.g., \cite{Beylkin, 
Novikov:att-inversion,Quinto1980}). Among the most prominent
reconstruction formulas are those of filtered backprojection type
\cite{Beylkin,KLM,Natterer86} which have the following
 form 
\begin{equation}
\label{eq:reconstruction operators}
	B g=\Rstn P g, 
\end{equation}
where $g=\Rmu f$ and $P$ is a general pseudodifferential operator that \enquote{filters} the
data $g=\Rmu f$. For example, in case of the classical Radon transform
the use of filter $P=1/2\cdot(2\pi)^{1-n}(-\partial^2/\partial s^2)^{\frac{n-1}{2}}$ in
\eqref{eq:reconstruction operators} leads to an exact reconstruction formula, $f=R^{\ast}PRf$,
which is the basis for the standard filtered backprojection (FBP) algorithm
\cite{Natterer86}. Another prominent example is the so-called Lambda
reconstruction formula (employed in local tomography) which uses the filter
$P=1/2\cdot(2\pi)^{1-n}(-\partial^2/\partial s^2)^{n/2}$ in \eqref{eq:reconstruction
operators} for $n$ is even. In contrast to the FBP reconstruction operator, 
the Lambda reconstruction operator is local  in even dimensions. 
However, when $n$ is odd, the FBP reconstruction operator is local
itself.

In classical imaging setups the FBP type reconstruction operators
\eqref{eq:reconstruction operators} are usually applied to full (complete)
data. As mentioned above, some of those filters even lead to exact
reconstructions if the data are complete. In this paper, we consider the
problem of reconstructing $f$ from incomplete data by using reconstructions
operators \eqref{eq:reconstruction operators} with general filters.  More
precisely, we assume that $\Rmu f$ is given only for directions $\om$ in a
closed subset $A\subset \snm$ with nontrivial interior.  Thus, we deal with
the restricted (or limited data) generalized Radon transform defined as
\[\RmuA: = \chiar \Rmu,\]  where $\chiar$ denotes the
characteristic function of the data space $A\times\R$ with the limited
angular range $A.$ Such limited data problems arise in many practical
situations and the filtered backprojection type reconstruction of the form
\eqref{eq:reconstruction operators} is still one of the preferred
reconstruction methods \cite{SidkyPanVannier2009} (where instead of the full
data $g=\Rmu f$, the limited data $g_A=\RmuA f$ is used for the
reconstruction). It is well known that the limited data reconstruction
problem is severely ill-posed \cite{Lo1986,Natterer86}. As a consequence,
only visible singularities can be reconstructed reliably \cite{Quinto93}
and additional artifacts can be generated,
cf. \cite{FrikelQuinto2013,Katsevich:1997}. In $\rtwo$, the geometry of
added artifacts has been precisely characterized in
\cite{FrikelQuinto2013,Katsevich:1997}.  In those articles, the authors
consider the classical limited angle FBP and Lambda reconstructions, i.e.,
$\mu=\nu\equiv 1$ and $P=1/(4\pi)\sqrt{-d^2/ds^2}$ for FBP and
$P=(1/4\pi)(-d^2/ds^2)$ for Lambda. In \cite{Katsevich:1997}, Katsevich
also considers general weights $\mu$ (and the dual transform w.r.t. weight
$1/\mu$) and the Lambda reconstruction operator.  In particular, the
authors of \cite{FrikelQuinto2013,Katsevich:1997} show that artifacts are
generated along straight lines that are tangent to singularities of $f$
whose directions correspond to the ends of the angular range. For the
classical Radon transform in $\rtwo$, the strength of added artifacts was
characterized by L.\ Nguyen in \cite{LNguyen1}.  In addition to
characterization of artifacts, the authors of \cite{FrikelQuinto2013,
  Katsevich:1997} show that the artifacts can be reduced by using modified
reconstruction operators. The same modified reconstruction operators are
considered in \cite{KLM} for $\Rmu$ with Lambda and FBP filters, and the
symbols are given for those specific operators for limited angle and ROI
data. In all of those cases, the calculation of the symbols relies on the
specific form of the filters.

This work is a generalization of the above mentioned results as it provides
a full characterization of visible singularities and added artifacts for
the restricted generalized hyperplane Radon transform in $\rn$ and for
reconstruction operators with general filters.  To prove these
characterizations we utilize a general paradigm that is based on the
calculus of Fourier integral operators and microlocal analysis.  This was
originally developed in \cite{FrikelQuintoPreprint, FrikelQuinto2014}, and
in \cite{FrikelQuinto2014} it was used to characterize artifacts in
photoacoustic tomography and sonar (see also \cite{Nguyen-spherical} for 
related results).
This methodology significantly different from techniques used in
\cite{FrikelQuinto2013,Katsevich:1997} (which rely on explicit expressions
of the reconstruction operators as singular pseudodifferential
operators). The flexibility of this approach allows us to prove
characterizations for reconstruction operators with general filters as well
as general weights.

In the case of $\rtwo$, our characterization (cf. Corollary
\ref{cor:characterization-rtwo}) of visible and added singularities
are in accordance with the results of \cite{FrikelQuinto2013,
Katsevich:1997}. However, our result is more general as it is valid
for reconstruction operators with general filters and weights, and it
provides conditions on filters and weights which guarantee the
recoverability of almost all visible singularities. In addition to
that, we also prove characterizations in the general case of the
hyperplane transform in $\rn$ in Theorem \ref{thm:characterization}.
To the best of our knowledge this is the first characterization of
added artifacts and visible singularities in $\rn$. In this paper, we
also define modified versions of the reconstruction operators
according to \cite{FrikelQuinto2013,FrikelQuinto2014,
Katsevich:1997,KLM} and prove that for general filters $P$ that they
do not add artifacts to the reconstruction (Theorem
\ref{thm:reduction1}).  Furthermore, we calculate the symbol of our
general reconstruction operators as pseudodifferential operators
(Theorem \ref{thm:reduction1}), and provide conditions for filters
under which the reconstructions operators are elliptic (Theorem
\ref{thm:symbol elliptic}). If the filters are chosen according to
those conditions, we show that classical as well as modified
reconstruction operators reliably recover almost all visible
singularities (Theorem \ref{thm:characterization} and
\ref{thm:reduction2}, respectively).

The article is organized as follows. Basic definitions and notations
are given in Section \ref{sec:notation}. In Section \ref{sec:paradigm} 
we present a general paradigm to characterize added singularities 
in limited data tomography. In Section \ref{sec:characterization-r2}
we first present the characterizations of the limited angle artifacts for the
generalized Radon transform in $\rtwo$. The generalization of these
results to $\rn$ is stated in Section \ref{sec:characterization-rn}, 
and the artifact reduction strategy as well as symbol calculations are 
given in Section \ref{sec:reduction}. The proofs are 
presented in the appendix.


%% file: 2_notation.tex
\section{Notation}
\label{sec:notation}


Let $\Om$ be an open set.  We denote the set of $C^\infty$ functions
with domain $\Om$, by $\Ec(\Om)$ and the set of $C^\infty$ functions
of compact support in $\Om$ by $\Dc(\Om)$.  Distributions are
continuous linear functionals on these function spaces.  The dual
space to $\Dc(\Om)$ is denoted $\Dc'(\Om)$ and the dual space to
$\Ec(\Om)$ is denoted $\Ec'(\Om)$. In fact, $\Ec'(\Om)$ is the set of
distributions of compact support in $\Om$.  For more information
about these spaces we refer to \cite{Rudin:FA}.

We will use the framework of microlocal analysis for our
characterizations. Here, the notion of a wavefront set of a
distribution $f\in\Dc'(\Om)$ is central. It simulta\-neously describes
the locations and directions of singularities of $f$. That is, $f$ has
a singularity at $x_0\in\Om$ in direction $\xio\in \rn\smo$ if for any
cutoff function $\vp$ at $x_0$, the Fourier transform $\Fc(\vp f)$
does not decay rapidly in any open conic neighborhood of the ray
$\{t\xio\st t>0\}$.  Then, the \emph{wavefront set} of
$f\in\Dc'(\Om)$, $\WF(f)$, is defined as the set of all tuples
$(x_0,\xio)$ such that $f$ is singular at $x_0$ in direction $\xio$.
As defined, $\WF(f)$, is a closed subset of $\rn\times(\rn\smo)$ that
is conic in the second variable. However, in what follows, we will
view the wavefront set as a subset of a cotangent bundle so it will be
invariantly defined on manifolds \cite{Treves:1980vf}.





We recall that, for a manifold $\Xi$ and $y\in \Xi$, the cotangent
space of $\Xi$ at $y$, $T^*_y(\Xi)$ is the vector space of all first
order differentials (the dual vector space to the tangent space
$T_y(\Xi)$), and the cotangent bundle $T^*(\Xi)$ is the vector bundle
with fiber $T^*_y(\Xi)$ above $y\in \Xi$.  That is
$T^*(\Xi)=\sparen{(y,\eta)\st y\in \Xi, \eta\in T^*_y(\Xi)}$. The
differentials $\dx_1$, $\dx_2,\dots$, and $\dx_n$ are a basis of
$T^*_x(\rn)$ for any $x\in \rn$.  For $\xi\in \rn$, we will use the
notation \[\xi\dx = \xi_1\dx_1+\xi_2\dx_2+\cdots+\xi_n\dx_n\in
T^*_x(\rn).\] If $\phi\in \rr$ then $\dphi$ will be the differential
with respect to $\phi$, and differentials $\dr$ and $\ds$ are defined
analogously.  

For the Radon transform in $\rn$, we introduce some more notation.  For
$\om\in\snm$ we define \bel{def:P} \Pw:\rn\to H(\om,0),\quad \Pw(x) =
x-(x\cdot \om)\om.\ee So, $\Pw(x)$ is the orthogonal projection of $x$ onto
this hyperplane. Note that \[\Pw(x) = \pi_{-\om}(x)\ \ \forall x\in \rn.\]
Then, $\Pw(x)\d \om$ is the covector in $T^*_\om(\snm)$ dual to the vector
$\Pw(x)\in H(\om,0)$ (where we have identified this hyperplane with the
tangent space $T_\om(\snm)$).

Let $X$ and $Y$ be manifolds, and $C\subset T^\ast (Y)\times T^\ast (X)$,
then
\begin{equation}
	C^t = \sparen{(x,\xi;y,\eta)\st  (y,\eta;x,\xi)\in
	C}.\label{def:At}
\end{equation}
If $D\subset T^*(X) $, we define
\bel{def:circ}
	C\circ D = \sparen{(y,\eta)\in T^\ast(Y)\st  \exists (x,\xi)\in
	D\st (y,\eta;x,\xi)\in C}.
\ee
Furthermore, if $\Pi_L:C\to T^*(Y)$ and $\Pi_R:C\to T^*(X)$ are the
natural projections, then \bel{projection-circ}
C\circ D = \Pi_L\paren{\Pi_R\inv\paren{D}}.\ee

Fourier integral operators (FIO) are linear operators on distribution
spaces that precisely transform wavefront sets.  They are defined in
\cite{Ho1971, Treves:1980vf} in terms of amplitudes and phase
functions.  If $X$ and $\Xi$ are manifolds and  $\Fc:\Dc'(X)\to
\Dc'(\Xi)$ is a FIO, then associated to $\Fc$ is the \emph{canonical
relation}  $C\subset T^*(\Xi)\times T^*(X)$.  Then the
H\"ormander-Sato Lemma (e.g., 
\cite[Th.\ 5.4, p.\
461]{Treves:1980vf}) asserts for $f\in \Ec'(X)$ that 
\begin{equation}
\label{thm:HS}
	\WF(\Fc f)\subset C\circ\WF(f).
\end{equation}

%% file: 3_paradigm.tex
\section{The paradigm} \label{sec:paradigm}


In this section, we will present a methodology that can be used to
prove characterizations of visible singularities and limited view artifacts for a number of
tomography problems.  In the next section, we will apply them to
$\Rmu$.

This methodology was originally developed in \cite{FrikelQuintoPreprint,FrikelQuinto2014}
and in \cite{FrikelQuinto2014} it was used
to understand visible and added singularities in limited data
photoacoustic tomography and sonar.  Denote the forward operator by
$\Mc:\Ec'(\Omega)\to\Ec'(\Xi)$ and assume $\Mc$ is a FIO.  The
\emph{object space} $\Omega$ is a region to be imaged and the
\emph{data space} $\Xi$ is a space that parameterizes the data.  A
\emph{limited data problem} for $\Mc$ will be a specification of a
closed subset $B\subset \Xi$ on which data are given, and in this case,
the limited data operator can be written \bel{MB} \MB f = \chib
\Mc,\ee where $\chib$ is the characteristic function of $B$ and the
product just restricts the data to the set $B$.  In the cases we
consider, the reconstruction operator is of the form
\bel{MB-reconstruction}\Mst P\MB,\ee where $\Mst$ is an appropriate
dual or backprojection operator to $\Mc$, and this models our
reconstruction operator \eqref{eq:reconstruction operators}.

 Our next theorem tells what
multiplication by $\chib$ does to the wavefront set.  It is a special case
of Theorem 8.2.10 in \cite{Hoermander03}.

 \begin{theorem}\label{thm:WF mult} Let $u$ be a distribution and let
$B$ be a closed subset of $\,\Xi$ with nontrivial interior.  If the
\emph{non-cancellation condition} 
\bel{non-cancellation}
\forall\,(y,\xi)\in \WF(u),\; (y,-\xi)\notin \WF(\chia) \ee holds,
then the product $\chia u$ can be defined as a distribution.  In this
case, we have \bel{WF of a product}\WF(\chib u) \subset
\Qc(B,\WF(u)),\ee where for $W\subset T^*(\Xi)$
\bel{def:Q}\begin{aligned}\Qc(B,W) :=& \big\{(y,\xi+\eta)\st y\in
B\,, \bparen{(y,\xi)\in W\text{\rm\ or } \xi = 0}\\&\qquad
\text{\rm\ and } \big[(y,\eta)\in \WF(\chib)\text{\rm\ or } \eta =
0\big]\big\}\,.\end{aligned}\ee
\end{theorem}

Note that the condition ``$y\in B$'' is not in \eqref{def:Q} in
H\"{o}rmander's theorem, but we  include this  because
$\chib$ is zero (and so smooth) off of $B$ (this is why we assume
$B$
is closed, so that its complement is open).  Also, note that the case
$\xi=\eta = 0$ in the definition of $\Qc$ is not allowed since the
wavefront set does not include the zero covector.


Our paradigm for proving characterizations for visible and added artifacts is given by the following procedure, cf. \cite{FrikelQuinto2014}:
\begin{enumerate}[(a)]

\item\label{paradigm:step1} Confirm the forward operator $\Mc$ is a
FIO and calculate its canonical relation, $C$.

\item\label{paradigm:step2} Choose the limited data set $B\subset
\Xi$ and calculate $\WF(\chib)$.

\item\label{paradigm:step3} Make sure the non-cancellation condition
\eqref{non-cancellation} holds for $\chib$ and $\Mc f$.  This can be
done in general by making sure it holds for $(y,\xi)\in C\circ
\paren{T^*(\Om)\smo}$. 

\item\label{paradigm:step4} Calculate $\Qc(B,C\circ \WF(f))$.

\item\label{paradigm:step5} Calculate $C^t\circ\Qc\paren{B,
C\circ\WF(f)}$ to find possible visible singularities and added
artifacts using \cite[Lemma 3.2]{FrikelQuinto2014}: \bel{final WF
containment}\WF(\Mst P \MB f)\subset C^t\circ\Qc\paren{B,
C\circ\WF(f)}.\ee
\end{enumerate}

The paradigm does not provide a lower bound for $\WF(\Mst P \MB f)$ in
terms of $\WF(f)$, but we will analyze the operators more completely
to provide such a bound.

%% file: 4_artifactsR2.tex

\section{Characterization of visible singularities\\ and added
artifacts in $\rtwo$}
\label{sec:characterization-r2} 

In this section, we present a characterization for the line transform
$\rtwo$ which is a special case of the more general results for the
hyperplane transform in $\rn$ that will be stated in the next section and
proven in the appendix. The statement is simpler in $\rtwo$. Here, we use a
slightly different (more convenient) notation and discuss some
implications. The presented characterization generalize the results of
\cite{FrikelQuinto2013,Katsevich:1997} where the authors consider specific
filters.


%

\subsection{The setup in $\rtwo$}
\label{subsec:notation-rtwo}
To make the presentation simpler, we will parametrize a line in $\rtwo$ 
in terms of an angle and a signed distance to the origin, as opposed to 
the parametrization \eqref{def:H}. To that end, we let $s\in \rr$,
$\phi\in[0,2\pi]$ and $\th(\phi) = (\cos(\phi),\sin(\phi))$ be the
unit vector in $S^1$ in direction $\phi$ and $\thperp(\phi)=(-\sin
(\phi),\cos(\phi))$, then $\thperp(\phi)$ is perpendicular to
$\th(\phi)$.  Let $\Xi = [0,2\pi]\times \rr$, then for each
$(\phi,s)\in \Xi$,  
\[L(\phi,s) = \sparen{x\in \rtwo\st x\cdot\th(\phi) = s}\] 
is the line containing $s\th(\phi)$ and normal
to $\th(\phi)$.  We let $\mu(\phi,x)$ be a smooth function on
$\rr\times \rtwo$ that is $2\pi-$periodic in $\phi$.  Then, the generalized 
Radon transform can be written as
\bel{def:Rmu-rtwo} \Rmu f(\phi,s) =\int_{x\in L(\phi,s)}f(x)\mu(\phi,x)\,\d x ,\ee 
where $\d x$ denotes the arc length measure on the line. This transform integrates
functions along lines. The corresponding dual transform (or the backprojection
operator) for $g\in \Sc(\sxr)$ and a smooth weight $\nu(\phi,x)$ then takes the form 
\bel{def:Rstarnu-rtwo} \Rstn g(x) =\int_0^{2\pi} g(\phi,x\cdot\th(\phi))\nu(\phi,x)\d\phi, \ee 
which is the integral of $g$ over all lines through $x$ (since for each
$\th(\phi)$, $x\in L(\phi,x\cdot \th(\phi))$). 
As noted in the introduction, these transforms are both defined and weakly
continuous for classes of distributions \cite{He:RT2011}.  

We consider the limited angle problem, i.e., we consider the data space 
of the form $[a,b]\times\R$ with the limited angular range $[a,b]$
where $b-a<\pi$ (or $b-a<2\pi$ if $\mu$ is not symmetric). 
Note that for $b-a\geq \pi$, every line can be parameterized by $\phi\in (a,b)$ 
although for general $\mu$, the weight might be different on the line
$L(\phi,s)$ and $L(\phi+\pi,-s)$: $\mu(\phi,x)$ might not equal
$\mu(\phi+\pi,x)$ for all $(\phi,x)$. Thus, we deal with the restricted
(or limited angle) generalized Radon transform which we define as 
\begin{equation}
\label{eq:restricted Radon transform}
	\Rmuab f(\phi,s) = \chi_{[a,b]\times\R}(\phi,s)\cdot\Rmu f(\phi,s),
\end{equation}
where $\chi_{[a,b]\times\R}$ denotes the characteristic function of
$[a,b]\times\R$. 

\begin{remark}
	In order to make clear how the characterization in $\rtwo$
follows from the more general statement in Theorem
\ref{thm:characterization}, we would like to point out that (in the
general statement) all terms involving $\omega\in S^{n-1}$ can be
stated in $\rtwo$ using the parametrization of $\om=\theta(\phi)\in
S^{1}$ with respect to the angle $\phi\in[0,2\pi]$ (see section
\ref{subsec:notation-rtwo}).  Hence, all terms that are generally
formulated with respect to $\d \om$ will be stated with respect to
$\d\phi$.  For example, the projection $\pi_{\omega}$ (cf.
\eqref{def:P}) simplifies to
$\pi_{\theta(\phi)}(x)=(x\cdot\thperp(\phi))\thperp(\phi)$ and, hence,
$\Pw(x)\d \om$ corresponds to $x\cdot\thperp(\phi)\d \phi$.
\end{remark}

\subsection{The characterization}

The corollary presented below is a special case of Theorem
\ref{thm:characterization} that is given in Section
\ref{sec:characterization-rn} and proven in the appendix.  The
statement is simpler in $\rtwo$ and the characterizations are in
accordance with the results in \cite{FrikelQuinto2013,Katsevich:1997}.
However, the following result is more general than the results of
\cite{FrikelQuinto2013,Katsevich:1997} as it is valid for
reconstruction operators with general filters and weights, and it
provides a lower bound \eqref{elliptic equality-rtwo} under an
ellipticity assumption.


%
%

To state the result in $\rtwo$, for $A\subset
\otp$, we define
\bel{def:VrA-rtwo} \VrA =\sparen{(\phi,s,\alpha[-z\d \phi + \d s])\st
\phi\in A, s\in \rr, z\in \rr, \alpha \neq 0}\ee
and let
\bel{def:VA-rtwo} \VA = \{(x,\alpha\th(\phi)\d x):
x\in \rtwo,\ \alpha\neq 0,\ \phi\in A\}. \ee
Now, for $f\in
\Ec'(\rn)$, define
\begin{equation}\label{def:WFA-rtwo} \WF_{A}(f)=\WF(f)\cap \VA.
\end{equation}
Similar definitions will be given for the transform in  $\rn$, and they
will reflect the notation used in that general case.

\begin{corollary}
\label{cor:characterization-rtwo} Let $\mu$ and $\nu$ be smooth
$2\pi$-periodic functions on $\rr\times \rtwo$. Let $P$ be a
pseudodifferential operator on $\Ec'(\Xi)$.  Let $f\in \Ec'(\rtwo)$.
Our limited data reconstruction operator is \[\Lab f = \Rstn P \Rmuab
f = \Rstn P \chiab \Rmu f.\] Then, 
\begin{equation} \label{equn:characterization-rtwo}
\WF(\Lab f)\subset \WF_{[a,b]}(f)\cup\Ac_{\{a,b\}}(f),
\end{equation} 
where $\WF_{[a,b]}(f)$ is defined according to \eqref{def:WFA-rtwo} and  
 \bel{def:A-rtwo}\begin{aligned} 
\Ac_{\{a,b\}}(f) = \{(x+t\th^\bot(\phi),&\alpha\th(\phi)\d x)\st \phi\in\sparen{a,b},\\ &
\alpha,t\neq0, (x,\alpha\th(\phi))\in\WF(f)\}
\end{aligned}\ee is the set of possible added artifacts.

Now, assume that $\mu$ and $\nu$ are nowhere zero and $b-a<\pi$.
Assume the top order symbol of $P$ is elliptic on $\Vrab$ defined by
\eqref{def:VrA-rtwo}.  Then, \bel{elliptic equality-rtwo} \WF_{(a,b)}(f)=
\WFab(\Lab f).\ee
\end{corollary}

The condition $b-a<\pi$ is reasonable in limited data
problems because, if $b-a> \pi$, then every line can be
parameterized by $L(\phi,s)$ for some $\phi\in (a,b)$ which amounts to
the full data problem.  

Corollary \ref{cor:characterization-rtwo} provides an upper bound for
the singularities of limited angle reconstructions: It is given as a
union of the visible wavefront set of $f$, $\WF_{[a,b]} f$
(singularities that actually belong to $f$), and the set of possibly
added artifacts $\Ac_{\{a,b\}}(f)$ (singularities that might be
artificially created by the reconstruction operators). In particular,
it shows that any singularity of $f$ with direction outside the given
angular range $[a,b]$ is smoothed by the limited angle reconstruction
operator $\Lab$. Those singularities cannot be reconstructed (are
invisible). The corollary also provides a precise geometric
description of possibly added artifacts. It shows that artifacts are
generated along straight lines that are normal to singularities of $f$
(e.g., tangent to boundaries of regions) whose directions correspond
to the ends of the angular range $\{a,b\}$. In other words, any
singularity of $f$ with direction $\theta(a)$ or $\theta(b)$,
generates added singularities along the line $L(a,x\cdot\theta(a))$ or
$L(b,x\cdot\theta(b))$, respectively.

Moreover, Corollary \ref{cor:characterization-rtwo} provides a
lower bound in form of the equation \eqref{elliptic equality-rtwo}
under an ellipticity assumption.  In particular, the equality
\eqref{elliptic equality-rtwo} guarantees that almost all visible
singularities will be reconstructed if the filters are chosen
appropriately. This statement is formulated for the case $b-a<\pi$
(limited angular range).

%% file: 4_artifactsRn.tex
\section{Generalizations to the Radon transform in $\rn$}
\label{sec:characterization-rn}

In this section, we present characterizations of visible singularities
and added artifacts for the restricted generalized Radon (hyperplane)
transform in $\rn$.  We analyze filtered backprojection type operators
with general filters and derive conditions for filters which guarantee
the recoverability of most of the visible singularities. Our results
generalize the characterizations given in
\cite{FrikelQuinto2013,Katsevich:1997}. We employ the paradigm
introduced in Section \ref{sec:paradigm}. In what follows, we use the
notation introduced in Section \ref{sec:introduction} and Section \ref{sec:notation}.


\subsection{Basic properties of $\Rmu$}


The first proposition provides the canonical relation of $\Rmu$ and
$\Rstn$ (by transpose), and this determines their microlocal
properties.

\begin{proposition}
\label{prop:canonical relations} If $\mu$ is smooth weight on
$\snm\times \rn$, then the generalized Radon transform $\Rmu$ is a
 Fourier integral operator associated to the canonical
relation \begin{multline} \label{def:C} C = \{
((\om,s),\alpha\bparen{-\Pw(x)\d\om+\d s}; x,\alpha\om\d x)\st \\
		 \om\in \snm,\alpha\neq 0, x\cdot \om = s
\}, \end{multline} and $(\om,x, \alpha)$ give coordinates on $C$
because $s=\om\cdot x$.  If $\mu$ is nowhere zero, then $\Rmu$ is
elliptic.

 If $\nu$ is smooth, then the backprojection operator $\Rstn$ is a
Fourier integral operator associated to the canonical relation $C^t$
defined in \eqref{def:At}, and if $\nu$ is nowhere zero, then $\Rstn$
is elliptic.  

Let $\Pi_R:C\to T^*(\rn)$ and $\Pi_L:C\to T^*(\Xi)$ be the natural
projections. Then $\Pi_L$ is an injective immersion and $\Pi_R$ is a
two-to-one immersion.  Let $(x,\xi\dx)\in T^*(\rn)\smo$. 
Define
\bel{def:lambdas}\begin{aligned}
\om(\xi)&=\xi/\norm{\xi}\in \snm\\
 \lamo(x,\xi) &=
\paren{\om(\xi),x\cdot \om(\xi), \norm{\xi}\bparen{-\Pw(x)\d \om + \d s}}\\
\lamon(x,\xi) &= \paren{-\om(\xi),-x\cdot \om(\xi),
-\norm{\xi}\bparen{-\Pw(x)\d\om + \d
s}}\end{aligned}\ee
where $\Pw$ is defined by \eqref{def:P}.

The two preimages of $(x,\xi\dx)$ under $\Pi_R$ are
\[(\lamo(x,\xi);x,\xi\dx)\ \ \text{ and }\ \
(\lamon(x,\xi);x,\xi\dx). \]
Therefore, \bel{lambdas}\begin{aligned}C\circ\sparen{(x,\xi\dx)}&=
\sparen{\lamo(x,\xi),\lamon(x,\xi)}\\
C^t\circ\sparen{\lamo(x,\xi\dx)}&=
C^t\circ\sparen{\lamon(x,\xi)}=\sparen{(x,\xi\dx)}.\end{aligned}\ee
\end{proposition}

\begin{proof}  The calculation of $C$ is well known, see e.g., 
 \cite{GS1977, Quinto1980}.  The generalized hyperplane transform
$R_\nu$ has the same canonical relation as $\Rmu$ since the weight
does not affect the canonical relation, only the symbol.  Since
$\Rstn$ is the dual of $R_\nu$, it is an FIO associated to $C^t$ by
the standard calculus of FIO, e.g., \cite[Theorem 4.2.1]{Ho1971}.
That $\Pi_L:C\to T^*(\Xi)\smo$ is an injective immersion (The Bolker
Assumption) is a straightforward calculation \cite{GS1977,
Quinto1980}.

One uses \eqref{def:C} to find the two preimages of $(x,\xi\dx)$ under
$\Pi_R:C\to T^*(\rn)\smo$ using the fact that $\xi =
\norm{\xi}\om(\xi)=-\norm{\xi}(-\om(\xi))$.  Statement \eqref{lambdas}
follows from the observation that, if $A\subset T^*(\rn)$, then $C\circ
A = \Pi_L\paren{\Pi_R\inv (A)}$, and if $B\subset T^*(\Xi)$, then
\eqref{projection-circ} can be used to show that $C^t\circ B =
\Pi_R\paren{\Pi_L\inv (B)}$ (where $\Pi_R$ and $\Pi_L$ are the maps
for $C$).
\end{proof}

\subsection{The limited data operators}

We are concerned with the limited data problem when $\Rmu f$ is given
only for angles $\om$ in a closed subset, $A$, of $\snm$ with
nontrivial interior.  The limited data set from Section
\ref{sec:paradigm} is 
\[B = A\times \rr.\]
We assume $A$ has nontrivial interior so that $\chiar$ is not the zero
distribution.

Our next proposition shows that the limited data forward operator and
our reconstruction operator are defined for distributions.

\begin{proposition}\label{prop:composition} Let $\mu$  and $\nu$ be  smooth
functions on $\snm\times \rn$. Let $P$ be a pseudodifferential
operator on $\Ec'(\Xi)$ and let $A\subset \snm$ be a closed set with
nontrivial interior. Let $f\in \Ec'(\rn)$.  Then, the limited data
forward operator for data on $A\times \rr$, \bel{def:RmuA} \RmuA: =
\chiar \Rmu,\ee maps $\Ec'(\rn)$ to $\Ec'(\Xi)$.

  The limited data reconstruction operator \bel{def:LA} \LA: = \Rstn P
\RmuA \ee maps $\Ec'(\rn)$ to $\Dc'(\rn)$.  Here $\Rstn$ is defined by
\eqref{def:Rstn-rn}.
\end{proposition}

This proposition will be proven in the appendix as a part of the proof
of Theorem \ref{thm:characterization}.

%

To describe the ellipticity conditions in our theorems, we need to
define the following sets.

\begin{definition}\label{def:sets}  Let $A\subset \snm$. 
Define
\bel{def:VrA} \VrA =\sparen{(\om,s,\alpha[-z\d \om + \d s])\st
\om\in A, s\in \rr, z\in H(\om,0), \alpha \neq 0}\ee
and let
\bel{def:VA} \VA = \{(x,\alpha\om\d x):
x\in \rn,\ \alpha\neq 0,\ \om\in A\}. \ee
Now, for $f\in
\Ec'(\rn)$, define
\begin{equation}\label{def:WFA} \WF_{A}(f)=\WF(f)\cap \VA.
\end{equation}\end{definition}

In the next section, we will show that, if $P$ is elliptic on $\VrA$, then
our reconstruction operator will recover almost all visible singularities,
and we will prove that the set $\VA$ will contain singularities of the
object that are visible in the data $\Rmu f$.

\subsection{The characterization}

The next theorem provides a characterization of visible singularities and 
added artifacts in arbitrary dimensions (using reconstruction operators with 
arbitrary filters $P$), and it also provides a lower bound \eqref{elliptic equality} 
under an ellipticity assumption. 
To state the result, we again let $A\subset \snm$ and we define \[(-1)A = \sparen{\om\in \snm\st -\om\in
A}.\]  Note that $\intt(A)$ is the interior of $A$, $\bd(A)$ is its
boundary, and $\cl(A)$ is its closure in $\snm$.

 \begin{theorem}\label{thm:characterization} Let $\mu$ and $\nu $ be a smooth
functions on $\snm\times \rn$. Let $P$ be a pseudodifferential operator
on $\Ec'(\Xi)$.  Let $A\subset \snm$ be a closed set with nontrivial
interior and let $\LA$ be defined by \eqref{def:LA}.  Then,
\begin{equation} \label{equn:characterization}
 \WF(\LA f)\subset \WF_{A}(f)\cup\Ac_{\bd(A)}(f),
\end{equation} 
where $\WFA(f)$ is defined in \eqref{def:WFA} and
\bel{def:A}\begin{aligned} \Ac_{\bd(A)}(f) =&
 \{(x+ty,\alpha\om\d
x)\st \om\in\bd(A),\ t\neq 0,\ \alpha\neq 0,\\ 
&\quad (x,\alpha\om\dx)\in\WF(f),\ y\in H(\om,0),\\
&\quad\quad \text{and } \ (\om, y\d\om)\in
\WF(\chia)\}
\end{aligned}\ee is the set of possible added artifacts.  

Now, assume that $\mu$ and $\nu$ are both strictly positive and the
top order symbol of $P$ is elliptic on $\VrA$ defined by
\eqref{def:VrA}.  Assume either 
\begin{enumerate}[(i)]

\item\label{cond:not symmetric} the following \emph{non-symmetry
condition} holds: \bel{not symmetric}\forall \om\in A,\ \ -\om\notin
A\ee 

or 

\item\label{cond:positive} the symbol of $P$ is real and is either
always positive or always negative on $\VrA$ and $A$ is symmetric
(that is $A=(-1)A$).

\end{enumerate}
Then, \bel{elliptic equality}
\WF_{\intt(A)}(f)= \WF_{\intt(A)}(\LA f).\ee
\end{theorem}

Note that, either the non-symmetry condition \eqref{not symmetric} or
the symmetry condition in \eqref{cond:positive} on $A$ is needed.  To
see this, assume there is a vector $\omo\in \intt(A)$ for which
$-\omo\in \bd(A)$.  Then, an added artifact caused by a covector
$(x,(-\alpha)(-\omo)\dx)\in \WF(f)$ and a singularity at
$(-\omo,y\d\om)\in \WF(\chia)$ can create added artifacts at points
$(x+ty,\alpha\om\dx)$.

The proofs of this and the other main theorems are in the appendix.

\begin{remark}\label{remark:thm:characterization} 
This theorem provides both upper and lower bounds for the
singularities that can be reconstructed at the limited angular range
$\om \in A$.  The upper bound is given as the union of the set of
visible singularities $\WF_A(f)$ and added artifacts $\Ac_{\bd(A)}(f)$
by \eqref{equn:characterization}. This shows that if a singularity of
$f$ is not in $\VA$, then it is smoothed by $\LA$.  This is reflected
by the fact that $\WF(\RmuA f)\subset \VrA$, which can be proven using
Proposition \ref{prop:canonical relations}.

The lower bound is given by the equality \eqref{elliptic equality}
under an ellipticity assumption.  In particular, the equality
\eqref{elliptic equality} provides a guarantee that almost all visible
singularities will be reconstructed if the filters are chosen
appropriately. 
\end{remark}

\begin{remark}
Radon transforms detect singularities conormal to the set being
integrated over (e.g., \cite{GS1977,Palamodov,Quinto93}), and the
above theorem states this relation explicitly: only singularities
$(x,\alpha\om\d x)\in\WF(f)$ with directions in the visible angular
range, $\om\in A$ (i.e., in $\Vc_{A}$) can be reconstructed
from limited data.  Singularities of $f$ at covectors outside $\VA$
are smoothed.  

The added singularities occur in the following way.  Each singularity
of $f$ in a direction $\om\in \bd(A)$ is spread along one or more
lines.  If $(x,\alpha\om\dx)\in \WF(f)$, then singularities are spread
in the hyperplane $H(\om,x\cdot\om)$.  For each $y\in H(\om,0)$ with
$(\om,y\d\om)\in \WF(\chia)$, singularities are spread along the line
in $H(\om,x\cdot\om)$ through $x$ and parallel to $y$. So, if $\bd(A)$
is smooth at $\om$, then there is only one line of singularities
(because the only singularities come from vectors $y$ that are normal
to $\bd(A)$ at $\om$).  

However, if $\bd(A)$ is not smooth at $\om$, then the singularity at
$(x,\alpha\om\dx)$ is spread on other lines.  For example, if
$\bd(A)$ has a corner, then for every $y\in H(\om,0)\smo$,
$(\om,y\d\om)\in \WF(\chia)$ so singularities are spread along the
whole hyperplane $H(\om,x\cdot\om)$.
\end{remark}

%% file: 5_reduction.tex
\section{Reduction of artifacts and preservation of visible singularities}
\label{sec:reduction}

In previous sections we have shown that the added artifacts are
generated due to the hard truncation at the boundary of the angular
range in the limited data generalized Radon transform \eqref{eq:restricted Radon transform}  
and \eqref{def:RmuA}, respectively.
In this section, we define a modified version of the reconstruction
operators according to \cite{FrikelQuinto2013,FrikelQuinto2014,
Katsevich:1997,KLM} that replace the sharp cutoff
$\chiar$ by a smooth cutoff and we prove that for general
filters $P$ they do not add artifacts to the reconstruction.

Let $\vp$ be a smooth cutoff function supported in $A$. We replace
$\chiar$ by $\vp$ in the reconstruction operator and define the
modified (artifact-reduced) reconstruction operator as \bel{def:L}
\Lvp f = \Rstn P\Kvp \Rmu f \ \text{ where } \ \Kvp g(\om,s) =
\vp(\om)g(\om,s)\ee (and where $\mu$ and $\nu$ are smooth weights).
This method was analyzed for the line transform in $\rtwo$ and for the
lambda filter $P=-d^2/ds^2$ and the FBP filter $P=\sqrt{-d^2/ds^2}$
and with $R_1$ in \cite{FrikelQuinto2013} and with $\Rmu$ (and
backprojection $R^*_{1/\mu}$) in \cite{Katsevich:1997, KLM}. Our
theorems provide generalization to $\rn$ and to arbitrary filters $P$,
and they provide the symbol of $\Lvp$ in general.

\begin{theorem}\label{thm:reduction1}  
 Let $\mu$ and $\nu$ be  smooth weights and let $\vp$ be a smooth function
 supported in $A$.  Then \bel{right reduction containment} \WF(\Lvp
 (f))\subset \WF_{A}(f).\ee The top order symbol of $\Lvp$ is
 \bel{symbol}\begin{aligned} \sigma(\Lvp)(x,\xi\dx) =&
  \frac{(2\pi)^{n-1}}{\norm{\xi}^{n-1}} \Big[
    \vp(\om(\xi))p(\lamo(x,\xi))\nu(\om(\xi),x)\mu(\om(\xi),x)\\
    &\qquad+\vp(-\om(\xi))p(\lamon(x,\xi))\nu(-\om(\xi),x)\mu(-\om(\xi),x) \Big]
\end{aligned}\ee where $p$ is the top order symbol of the
pseudodifferential operator $P$ and the other notation is given in
\eqref{def:lambdas}.
\end{theorem}

The proofs of this and the other main theorems are in the appendix.

The containment \eqref{right reduction containment} shows that the
modified reconstruction operators reconstruct only visible singularities
and, hence, do not add artifacts. This result provides an upper bound
for the visible singularities that are reconstructed through the modified
reconstruction operators $\Lvp$. In the following theorem, we also establish a
lower bound for the visible singularities under an ellipticity assumption
on the reconstruction operators.

\begin{remark}\label{remark:weights}  We now discuss two special
weights.  If $\nu = \mu$, then $\Rstn$ is the formal adjoint of $\Rmu$
and $\nu(\om(\xi),x)\mu(\om(\xi),x)$ is replaced by
$\mu^2(\om(\xi),x)$  in the symbol  of $\Lvp$.

If $\mu$ is nowhere zero and $\nu=1/\mu$, then the symbol of $\Lvp$ is
especially simple: \[\sigma(\Lvp)(x,\xi\dx) =
\frac{(2\pi)^{n-1}}{\norm{\xi}^{n-1}}
\bparen{\vp(\om(\xi))p(\lamo(x,\xi))+\vp(-\om(\xi))p(\lamon(x,\xi))},
\]  and so the top order symbol  of $\Lvp$ is not affected by the
weight $\mu$ and the only $x$ dependence comes from the choice of $P$,
as opposed to the general case with arbitrary $\mu$ and
$\nu$.\end{remark}

\begin{theorem}\label{thm:reduction2}   Let $\vp$
  be a nonnegative smooth function supported on $A$ and nonzero on
$\intt(A)$ and let $\mu$ and $\nu$ be smooth positive weights.  Assume
the symbol $\sigma(\Lvp)$ in \eqref{symbol} is elliptic on $\VintA$
defined in \eqref{def:VA}.  Then, \bel{left reduction
containment}\WF_{\intt(A)}(f) \subset \WF(\Lvp (f))\subset
\WF_{A}(f).\ee This implies that \bel{left equality}\WF_{\intt(A)}(f)
= \WF_{\intt(A)}(\Lvp (f)).\ee
\end{theorem}

This theorem shows that, as long as the filter $P$ is well-chosen, almost all
visible wavefront directions (those in $\WF_{\intt(A)}(f)$) are visible
using the artifact reduced operator $\Lvp$ and artifacts are not added
since $\WF( \Lvp(f))$ is contained in $\WF_{A}(f)$.

%

Our next theorem provides conditions on the filters $P$ that guarantee
the ellipticity of the reconstruction operators $\Lvp$. In particular,
it specifies some cases in which Theorem \ref{thm:reduction2} can be
applied.

\begin{theorem}\label{thm:symbol elliptic}
Let $\vp$ be a nonnegative function supported in $A$ and nonzero on
$\intt(A)$.  Assume that $\mu$ and $\nu$ are smooth and strictly
positive and the top order symbol of $P$ is elliptic on $\VrA$ defined
by \eqref{def:VrA}.  Assume either 
\begin{enumerate}[(i)]

\item\label{b-a} The following \emph{non-symmetry
condition} holds: \bel{not symmetric-sect5}\forall \om\in A,\ \ -\om\notin
A, \ee 

\item\label{same sign} or the symbol of $P$ is real and either always
positive or always negative on $\VrA$.

\end{enumerate}
Then $\Lvp =\Rstn \Kvp P \Rmu$ is elliptic on $\Vc_{\intt(A)}$
(defined by \eqref{def:VA}). 
 Therefore,
\eqref{left reduction containment} and \eqref{left equality} hold.
\end{theorem}

Note that condition \eqref{not symmetric-sect5} is the same condition,
\eqref{not symmetric}, used in Theorem \ref{thm:characterization},

\begin{example} We now discuss these conditions for the Radon line
transform in the plane and for $\mu=\nu\equiv1$.

First, we consider two standard filters, $P$.  Condition \eqref{same
sign} holds, for example, if $P=-d^2/ds^2$, the filter in Lambda
tomography, or $P=\sqrt{-d^2/ds^2}$, the filter in FBP because, in
both cases, the symbol is of the same sign on $\VrA$ (e.g.,
$\sigma(\sqrt{-d^2/ds^2})(\om,s,\beta\d\om + \alpha\d s)=\abs{\alpha}$
is real and never zero), and our theorem can be applied to these
operators.

Now, let $P=d/ds$.  If $A= \sparen{(\cos(\phi),\sin(\phi))\st \phi\in
[a,b]}$ and $b-a<\pi$, then condition \eqref{b-a} holds.  Since the
symbol of $d/ds$ is nowhere zero on $\VrA$, $\Lvp$ is elliptic on
$\Vc_{\intt(A)}$.

However, if $b-a>\pi$, ellipticity of $P$ is not sufficient for
ellipticity of $\Lvp$.  For example, consider the full data problem
for the classical transform $R_1$ and $P = (-i)d/ds$, then
$\sigma(P)(\om,s,\beta\d\om+\alpha\d s) = \alpha $ changes sign on
$\VrA$, even though $P$ is elliptic, and the operator $R_1^*((-i)d/ds
\,R_1)=0$ by symmetry.  Of course, the analogous point can be made for
$P=d/ds$.
\end{example}

%% file: 6_conclusion.tex
\section{Concluding remarks} 

In this work, we have characterized
visible singularities and added artifacts for the limited data
problem associated with the restricted generalized Radon (hyperplane) transform in $\rn$. In
particular, we analyzed filtered backprojection reconstruction
operators with general filters and proved that a simple modification
of these operators leads to an artifact reduction. To the best of our knowledge
this work for the first time provides characterizations of artifacts for the 
restricted generalized Radon transform in $\rn$ (Theorem \ref{thm:characterization}), which includes the classical
setup in $\rtwo$ as a special case (Corollary \ref{cor:characterization-rtwo}). However, even in the case of $\rtwo$ our results
are more general than the characterizations presented in \cite{FrikelQuinto2013,Katsevich:1997} 
since they are valid for general reconstruction operators with arbitrary filters. 
Our proofs use the general paradigm (originally developed in 
\cite{FrikelQuintoPreprint,FrikelQuinto2014}) that is based on the calculus of Fourier integral operators and
microlocal analysis. This technique is substantially different from the 
one used in \cite{FrikelQuinto2013,Katsevich:1997} where the authors use
explicit expressions of the considered (specific) reconstruction operators (and hence they
know the symbols of these operators explicitly).
We would like to note that the paradigm that we use in our proofs does not 
provide a way to characterize which part of the visible singularities will be reconstructed,
it enables us to derive only an upper bound for the wavefront set of
the limited data reconstructions. In fact, no lower bound can be derived 
for general reconstruction operators with arbitrary filters.
To guarantee that most of the visible singularities will
be reconstructed we need to make sure that the reconstruction
operators are elliptic. This can be done by choosing the filters
appropriately.  As one of our main results, in Theorem
\ref{thm:symbol elliptic}, we derive conditions for filters that
guarantee ellipticity of the filtered backprojection reconstruction
operators.  To that end, we calculate the symbol of the general
reconstruction operators in Theorem \ref{thm:reduction1}.

%% file: 7_acknowledgments.tex
\section*{Acknowledgments} The authors thank Frank Filbir for
encouraging this collaboration.  They thank Adel Faridani and
Alexander Katsevich for pointing out important references. They also
thank Venkateswaran Krishnan for valuable comments concerning this
work. They first author acknowledges support from the HC \O rsted
Postdoc programme, co-funded by Marie Curie Actions.  The work of the
second author is partially supported by NSF grant DMS 1311558.

%% file: 8_appendix.tex
\begin{appendix}
\section{Appendix}


We prove our main theorems in this appendix because the proofs are all
related.

 \begin{proof}[Proof of Proposition \ref{prop:composition} and Theorem
\ref{thm:characterization}] We use the paradigm presented in Section
\ref{sec:paradigm} to prove \eqref{equn:characterization}.  By
Proposition \ref{prop:canonical relations}, we know that $\Rmu$ is a
Fourier integral operator with the canonical relation given in
\eqref{def:C}. Thus, the step \eqref{paradigm:step1} of our paradigm is
carried out.

For the step \eqref{paradigm:step2}, we consider $B=A\times\R$ and
compute $\WF(\chiar)$.  Note that \bel{WF(chiar)}\WF(\chiar)=
\sparen{(\om,s;\eta+0\ds)\st (\om,\eta)\in \WF(\chia),\ s\in \rr}.\ee
Note that $\chia$ is smooth on the complement of $\bd(A)$, so
the covectors in $\WF(\chiar)$ all have $\om\in \bd(A)$.  

%

First, note that every covector in $C\circ \paren{T^*(\rn)}$ has
nonzero $\ds$ component by \eqref{def:C}. Therefore, every covector in
$\WF(\Rmu f)$ has nonzero $\ds$ component.  Since $\WF(\chiar)$ has
zero $\d s$-component, we see that the non-cancellation condition
\eqref{non-cancellation} holds. This is step \eqref{paradigm:step3} of
our paradigm.  Hence, by Theorem \ref{thm:WF mult}, the product $\RmuA
f = \chiar\Rmu f$ is well-defined as a distribution with compact
support since $f\in \Ec'(\rn)$.  This proves that $\RmuA:\Ec'(\rn)\to
\Ec'(\Xi)$, the first statement in Proposition
\ref{prop:composition}.  Since $P:\Ec'(\Xi)\to \Dc'(\Xi)$ and
$\Rstn:\Dc'(\Xi)\to\Dc'(\rn)$, the final statement of that Proposition
also holds.

Continuing the proof of Theorem \ref{thm:characterization}, we do the
next step, \eqref{paradigm:step4}, of the paradigm, and calculate
$\Qc(A\times\R,C\circ\WF(f))$ using \eqref{def:Q}.  By definition, the
set $\Qc(A\times\R,C\circ\WF(f))$ is a union of three sets:
\bel{eq:wavefront set restricted Radon transform}
\begin{aligned}
	\Qc(A\times\R,&C\circ\WF(f))  \\
&=\bparen{(C\circ\WF(f)) \cap
	\{((\om,s),\eta)\in T^\ast (\Xi)\st \om\in A\}}\\
 &\qquad\qquad\cup
	\WF(\chiar) \cup \WbdA(f),
\end{aligned}\ee
where the first set (in braces) corresponds to $\xi\neq 0$, $\eta=0$ in the
expression ``$(y,\xi+\eta)$'' in the definition of $\Qc$, \eqref{def:Q},
the second to $\xi=0$, $\eta\neq 0$ and the third, $\WbdA(f)$, corresponds
to $\xi\neq 0$, $\eta\neq 0$.

Note that points in this third set, $\WbdA(f)$, are sums of covectors
in $WF(\chiar)$ and covectors in $C\circ \WF(f)$ that have the same
base points.  The only way a common base point occurs is when
\begin{enumerate}[(i)]
\item \label{bd(A)} $\om\in \bd(A)$ and there is a $y\in H(\om,0)$
with $(\om,y\d\om)\in \WF(\chia)$, generating singularities in
$\WF(\chiar)$, and

\item \label{WF(f)} there is an $x\in \rn$ and $\alpha\neq 0$ so that
$(x,\alpha\om\dx)\in \WF(f)$, generating singularities in $C\circ
\WF(f)$.
\end{enumerate}
In this case, the common base point is $(\om,x\cdot\om)$.
Since $\chia$ is a real function, if $(\om,y\d\om)\in \WF(\chia)$,
then so is $(\om,ty\d\om)$ for any $t\neq 0$.  Therefore,
corresponding covectors for \eqref{bd(A)} in $\WF(\chiar)$ are given
by \bel{bd(A)-covector} \sparen{(\om,x\cdot\om; t y\d\om+0\ds) \st
t\neq 0,\ (\om,y\d\om)\in \WF(\chia)}\ee Similarly, the corresponding
covectors for \eqref{WF(f)} in $C\circ \WF(f)$ are
\bel{WF(f)-covector} \sparen{(\om,x\cdot\om;
t'\alpha(-\Pw(x)\d\om+\ds))\st t'>0} \ee ($f$ is not assumed to be
real, so $\WF(f)$ is only \emph{positive} homogenous in the cotangent
coordinate).

Adding covectors in \eqref{bd(A)-covector} and \eqref{WF(f)-covector}
gives covectors in $\WbdA(f)$. Putting all this together,
\bel{sum}\begin{aligned}W(\om,x,\alpha):=&\big\{\paren{\om,x\cdot\om;
-\alpha t'\paren{\bparen{\Pw(x)+ty}\d\om+\ds}}\st\\& \qquad t'>0,t\neq
0, (\om,y\d\om)\in \WF(\chia)\big\}\end{aligned}\ee is the subset of
$\WbdA(f)$ associated to $\om\in \bd(A)$ and each $\alpha\neq 0$
and each $x\in \rn$ such that $(x,\alpha\om\dx)\in \WF(f)$. Note that
we have rescaled $t$ in order to factor as indicated in \eqref{sum}.

To accomplish the step \eqref{paradigm:step5} in our paradigm, we let
$P$ be a pseudodifferential operator.  Then, by containment
\eqref{final WF containment},
\[\WF(\Rstn P\RmuA f)\subset
C^t\circ\Qc\paren{A\times\R,C\circ \WF(f)}.\]

We now compute $C^t\circ\Qc\paren{A\times\R,C\circ \WF(f)}$. Using
\eqref{eq:wavefront set restricted Radon transform} and the
composition rules, first observe that
\bel{eq:three parts of the wavefront set}\begin{aligned}
	C^t \circ \Qc(&A\times\R,C\circ \WF(f)) \\
&= C^t  \circ
	\big[(C\circ\WF(f))\cap \{((\om,s),\eta)\in T^\ast (\Xi)
\st \om\in A\}\big]\\ 
	 & \qquad \cup C^t  \circ \WF(\chiar)\\
&\qquad \qquad	  \cup C^t  \circ \WbdA(f).\end{aligned}\ee

We examine the three terms in the right side of the equation
\eqref{eq:three parts of the wavefront set} separately. First, we get
\begin{multline*}
	C^t  \circ \big[(C\circ\WF(f)) \cap
	\{((\om,s),\eta)\in T^\ast (\Xi)\st \om\in A\}\big] \\
	=\big[(C^t  \circ C )\circ\WF(f))\big] \cap \big[C^t \circ
	\{((\om,s),\eta)\in T^\ast (\Xi)\st \om\in A\}\big].
\end{multline*}
Because $\Pi_L$ is injective and $\Pi_R$ is surjective to
$T^*(\rn)\smo$, \[C^t\circ C =
\Delta :=\sparen{(x,\xi \d x;x,\xi\d x)\st (x,\xi\d x)\in
T^\ast\rn\smo}\] and $\Delta\circ\WF(f)=\WF(f)$. Furthermore,
\bel{visible1} C^t \circ \{((\om,s),\eta)\in T^\ast (\Xi)\st
\om\in A\} = \VA.\ee Hence, the first set in
\eqref{eq:three parts of the wavefront set} is equal to the set of
visible singularities  \eqref{def:WFA} \[\WF_{A} (f)=\WF(f)\cap
\VA.\]

For the second set in \eqref{eq:three parts of the wavefront set}
observe that $C^t\circ \WF(\chiar)=\emptyset$ since the
$\d s$-components of covectors in $\WF(\chiar)$ is zero
and the $\d s$-components of covectors in $C^t$ is always non-zero. 

Finally, we consider the set $C^t\circ \WbdA(f) $.  Let $\om\in
\bd(A)$ and let $x\in \rn$ and $\alpha\neq 0$ such that
$(x,\alpha\om\dx)\in \WF(f)$.  Then, \eqref{sum} gives 
the subset, $W(\om,x,\alpha) $, of  $\WbdA(f)$ associated to $\om$, $x$,
and $\alpha$.  For each  $y$ such
that $(\om,y\d\om)\in \WF(\chia)$, there are elements of
$W(\om,x,\alpha)$ for each $t\neq 0$ and $t'>0$:
\[\gamma = (\om,x\cdot\om; -\alpha
t'\paren{\bparen{\Pw(x)+ty}\d\om+\ds}).\] Then,
$C^t\circ\sparen{\gamma}$ is the covector $(x',\alpha\om\dx)$, such
that $x'\in H(\om,x\cdot \om)$ (since $x'\cdot\om = x\cdot \om$) and
such that \[\Pw(x') = \Pw(x) + ty\] (see the definition of $C^t$ and
\eqref{def:C}). Since $x'\in H(\om,x\cdot\om)$, \[x' = x+ty\] and since
$t$ is arbitrary, the set \bel{added sings}\sparen{(x+ty,
\alpha\om\dx)\st t\neq 0,\ (\om,y\d\om)\in \WF(\chia)}\ee is the set
of added singularities coming from $\om\in \bd(A)$, $y\in H(\om,0)$
such that $(\om,y\d\om)\in WF(\chia)$ and $x\in \rn$ and $\alpha\neq 0$
such that $(x,\alpha\om\dx)\in \WF(f)$.  This proves
\eqref{equn:characterization} and \eqref{def:A}.

\medskip

Containment \eqref{elliptic equality} is proven using Theorem
\ref{thm:symbol elliptic}, which is proven below.  Let $(x,\xi\dx) \in
\WF(f)\cap \Vc_{\intt(A)}$.  Then, at least one of the unit vectors
$\om(\xi)$ or $-\om(\xi)$ (defined in Proposition \ref{prop:canonical
relations}) is in $\intt(A)$.  Without loss of generality, assume
$\omo=\om(\xi)\in \intt(A)$.

First we consider case \eqref{cond:not symmetric}.  Let $\vp$ be a
smooth cutoff function in $\om$ that is supported in a small open set
$U\subset A$ and equal to one in a smaller neighborhood $U'$ of
$\omo$. Since $\cl(U)\subset A$, if $\om \in \cl(U)$ then $-\om\notin
\cl(U)$.

We define $\Kvp$ as the multiplication operator $\Kvp g(\om,s) =
\vp(\om)g(\om,s)$.

Let \[g_1=P\Kvp \Rmu(f), \quad g_2 = P\bparen{\chiar-\vp} \Rmu(f).\] By
Theorem \ref{thm:symbol elliptic} part \eqref{b-a} applied to the set
$\cl(U)$, the symbol of $\Rstn P \Kvp \Rmu$ is elliptic on $\Vc_{U}$
and so at $(x,\xi\dx)$.  Therefore, $(x,\xi\dx)\in \WF(\Rstn g_1 )$.

We now show $(x,\xi\dx)\notin \WF\paren{\Rstn g_2}$.  Because
$\bparen{\chi_{A\times\rr}- \vp}$ is zero on $U'\times \rr$,
$\bparen{\chi_{A\times\rr}- \vp}\Rmu f$ is zero on $U'\times \rr$.
Therefore, $g_2= P\bparen{\chi_{A\times\rr}- \vp} \Rmu(f))$ is smooth
on $U'\times \rr$.  Since $\om(\xi)\in U'$, $\lamo(x,\xi)\notin
\WF(g_2)$.  By the non-symmetry condition \eqref{not symmetric},
$-\om(\xi)\notin A$, so $g_2$ is zero and hence smooth near
$-\om(\xi)$. This implies that $\lamon(x,\xi)\notin \WF(g_2)$.  Using
the H\"ormander-Sato Lemma (see \eqref{thm:HS}) $\WF(\Rstn g_2)\subset
C^t\circ \WF (g_2)$, so, by \eqref{lambdas} the only two covectors,
$\lamo(x,\xi)$ and $\lamon(x,\xi)$, that can contribute to wavefront
of $\Rstn g_2$ at $(x,\xi\dx)$ are not in $\WF(g_2)$ so
$(x,\xi\dx)\notin \WF(\Rstn g_2)$.

Therefore, $(x,\xi\dx)\in \WF(\Rstn g_1+\Rstn g_2)=WF(\LA f)$, and
this proves the final part of the theorem in case  \eqref{cond:not
symmetric}.

Now we consider case \eqref{cond:positive} and note that the symbol of
$P$ is of the same sign on $\VrA$.  Let $\omo\in \intt(A)$.  Then,
$-\omo\in \intt(A)$ by the symmetry condition for this case.  We let
$U\subset A$ be a  neighborhood of $\omo$ small enough so that $U$ is disjoint
from $(-1)U$.  Since $A$ is symmetric, so is $\intt(A)$, and
$(-1)U\subset \intt(A)$.  Let $\tU=U\cup (-1)U$ and let $\vp$ be a
smooth, nonnegative, even function that is supported in $\tU$ and is
one in a smaller neighborhood, $U'$ of $\omo$ (and therefore in the
neighborhood of $(-1)U'$ of $-\omo$).  Let $g_1$ and $g_2$ be as
defined in the first part of the proof.  By Theorem \ref{thm:symbol
elliptic} case \eqref{same sign} applied on $\cl(\tU)$, the symbol of
$\Rstn P \Kvp \Rmu$ is elliptic on $\Vc_{U}$ and so at $(x,\xi\dx)$.
Therefore, $(x,\xi\dx)\in \WF(\Rstn g_1 )$.

We now show $(x,\xi\dx)\notin \WF\paren{\Rstn g_2}$.  Let $\tU'=U'\cup
(-1)U'$.  Because $\bparen{\chi_{A\times\rr}- \vp}$ is zero on
$\tU'\times \rr$, $\bparen{\chi_{A\times\rr}- \vp}\Rmu f$ is zero on
$\tU'\times \rr$.  Therefore, $g_2= P\bparen{\chi_{A\times\rr}- \vp}
\Rmu(f))$ is smooth on $\tU'\times \rr$.  Since $\om(\xi)\in U'\subset
\tU'$, $\lamo(x,\xi)\notin \WF(g_2)$.  For the analogous reason,
$\lamon(x,\xi)\notin \WF(g_2)$.  The final part of the proof continues
as for the case \eqref{cond:not symmetric} to conclude that
$(x,\xi\dx)\in \WF(\Rstn g_1+\Rstn g_2)=WF(\LA f)$.\end{proof}

We now give the proofs of Theorems
\ref{thm:reduction1}-\ref{thm:symbol elliptic}.

\begin{proof}[Proof of Theorem \ref{thm:reduction1}]
  We use the notation, conventions, and symbol calculation in
\cite[Theorem 3.1]{Quinto1980}.  Recall that $\Pi_R:C\to
T^*(\rn)\smo$ and $\Pi_L:C\to T^*(\Xi)\smo$ are the natural
projections.  Let
\[Z=\sparen{(\om,s,x)\st \om\in \snm, x\in
\rn, s=\om\cdot x},\] then $Z$ is the set in over which the Schwartz
kernel of $\Rmu$ integrates (e.g., \cite{Quinto1980}).  To define the
measures used in \cite{Quinto1980}, we define global coordinates for
$Z$:  \[(\om,x)\mapsto (\om,x\cdot\om,x).\] Of course, $(\om,s,x)$ are
global coordinates on $\Xi\times \rn$. The measure on $Z$
associated to $\Rmu$ is $\mu(\om,x)d\om\,d x $ (see equation (16) in
\cite{Quinto1980}).  Equation (14) in \cite{Quinto1980} and the
discussion below it give the symbol of $\Rmu$ as the half density
\bel{symbol Rmu}\sigma(\Rmu)=\frac{(2\pi)^{(n-1)/2}\mu(\om,x)d \om\, d
x\, \sqrt{ds\,d\eta}}{\sqrt{d\om\, d s \, d x
}\,\Pi_R^*(\abs{\sigma_\rn})}\ee where $\abs{\sigma_\rn}$ is the
density from the canonical symplectic form on $T^*(\rn)$ and
$\Pi_R^*(\abs{\sigma_\rn})$ is its pull back to $C$.  Finally $\eta$
is the fiber coordinate in the conormal bundle of $Z$.  A similar
proof shows that the symbol of $\Rstn$ is given by \bel{symbol
Rstn}\sigma(\Rstn)=\frac{(2\pi)^{(n-1)/2}\nu(\om,x)d \om\, d x
\sqrt{ds\,d\eta}}{\sqrt{ d\om\, d s\, d
x}\,\Pi_L^*(\abs{\sigma_\Xi})}.\ee The pseudodifferential operator
$P\Kvp$ has symbol $\vp(\om)p(\om,s,\gamma)$ (where $\gamma\in
T^*_{(\om,s)}(\Xi)$) so $P\Kvp \Rmu$ is a standard smooth FIO and its
top order symbol is
\[\sigma(P \Kvp
\Rmu)=\frac{(2\pi)^{(n-1)/2}p(\om,s,\gamma)\vp(\om)\mu(\om,x)d \om\, d x
\sqrt{ds\,d\eta}}{\sqrt{d\om \,d s \,d x }\,\Pi_R^*(\abs{\sigma_\rn})}\] when evaluated at
covectors on $C$.

 Let $(x,\xi\dx)\in T^*(\rn)\smo$.  To calculate the symbol of the
composition of $\Rstn$ with $P\Kvp \Rmu$ one uses the note at the top
of p.\ 338 of \cite{Quinto1980}: since the projection $\Pi_R:C\to
T^*(\rn)\smo$ is two-to-one, the symbol of $\Rstn P \Kvp \Rmu$ at
$(x,\xi\dx)\in T^*(\rn)$ is the sum of the product
$\sigma(\Rstn)\cdot\sigma(P\Kvp \Rmu)$ at the two preimages.  Those
preimages are given by $\Pi_R\inv(x,\xi\dx)$, and by Proposition
\ref{prop:canonical relations}, they are the two covectors
$(\lamo(x,\xi);x,\xi\dx)$ and $(\lamon(x,\xi);x,\xi\dx)$.

Under the conventions of \cite{Quinto1980}, the symbol of $\Rstn P\Kvp
\Rmu$ at $(x,\xi\dx)$ is the sum
\bel{symbol-start}\begin{aligned}\sigma(\Rstn P \Kvp \Rmu)&(x,\xi\dx)
= \sparen{\frac{(2\pi)^{n-1} (d \om\, \d x)^2 d s\,
d\eta}{d \om\, d s\,d x\,\Pi_R^*(\abs{\sigma_\rn})\,\Pi^*_L(\abs{\sigma_\Xi})}}\\
&\quad\times 
\big[\vp(\om(\xi))\nu(\om(\xi),x)\mu(\om(\xi),x)p(\lamo(x,\xi))\\
&\quad\quad+\vp(-\om(\xi))\nu(-\om(\xi),x)\mu(-\om(\xi),x)p(\lamon(x,\xi))\big]\end{aligned}\ee
Now, \cite[Lemma 3.2]{Quinto1980} shows, for the Radon line transform,
that the term on the top right in braces in \eqref{symbol-start} can
be  simplified to equal to $(2\pi)^{n-1}/\norm{\xi}^{n-1}$.  Putting this
into \eqref{symbol-start} proves the symbol calculation
\eqref{symbol}.
\end{proof}

\begin{proof}[Proof of Theorem \ref{thm:reduction2}]
Because the symbol of $\Lvp$ is elliptic on $\VintA$ by assumption, the
left hand containment \eqref{left reduction containment} follows from
e.g., \cite[Prop.\ 6.9]{Treves:1980vf}. 

We now prove that the right-hand containment in \eqref{left reduction
containment}.  First, note that $\WF(\Lvp f)\subset \WF(f)$ since $\Lvp$ is a
standard pseudodifferential operator.  Second, $\Lvp$ smooths outside
of $\Vc_{A}$ for the following reason.  Since $\vp$ is zero
outside of $\sparen{(\om,s)\st \om\in A}$, $\WF(\vp\Rmu f)\subset
V^R_{A}$.  Now, using \eqref{visible1} and the H\"ormander Sato
Lemma \eqref{thm:HS}, one shows that $\WF(\Lvp f)\subset C^t\circ \VrA
=\Vc_{A}$.
This proves the right hand containment. 
\end{proof}

\begin{proof}[Proof of Theorem \ref{thm:symbol elliptic}]
In each case, we will show that $\sigma(\Lvp)$ is elliptic on
$\Vc_{\intt(A)}$.  Let $(x,\xi\dx)\in \Vc_{\intt(A)}$, then either
$\om(\xi)$ or $-\om(\xi)$ or both are in $\intt(A)$. Without loss
of generality, we assume $\om(\xi)\in \intt(A)$.  Therefore,
$\vp(\om(\xi))\neq 0$.  

In case \eqref{b-a}, we assume $\mu$ and $\nu$ are smooth and nowhere
zero, and we assume that, if $\om(\xi)\in \intt(A)$, then
$-\om(\xi)\notin A$.  Therefore, one and only one term in brackets in
\eqref{symbol} is nonzero, and the symbol is elliptic on
$\VintA$.

In case \eqref{same sign}, we assume $\mu$ and $\nu$ are positive, the
top order symbol of $P$, $\sigma(P)=p$, is real, elliptic, and of the
same sign everywhere on $\VrA$.  Since $\vp\neq 0$ on $\intt(A)$ and
$\nu\mu>0$, at least the first term in brackets in \eqref{symbol} (the
one containing $\om(\xi)$) is nonzero.  The second term (containing
$-\om(\xi)$) either has the same sign as this term (since the sign of
$p$ does not change) or is zero (if $\vp(-\om(\xi))=0$).  Therefore
the sum is nonzero and so the symbol of $\Lvp$ is elliptic on
$\Vc_{\intt(A)}$.

In either case, we have concluded the symbol of $\Lvp$ is elliptic on
$\VintA$.  Now, one can use the conclusion of Theorem \ref{thm:reduction2}
to finish the proof.
\end{proof}

\end{appendix}